\documentclass[10pt,reqno]{amsart}
\usepackage{amsmath,amssymb,amsthm}
\usepackage{microtype} 
\usepackage{cite}
\usepackage{enumitem}

\usepackage{hyperref}
\usepackage{mathtools}
\mathtoolsset{showonlyrefs}

\allowdisplaybreaks

\newtheorem{theorem}{Theorem}[section]

\newtheorem{lemma}[theorem]{Lemma}

\theoremstyle{definition}
\newtheorem{remark}[theorem]{Remark}

\newcommand{\C}{\mathbb{C}}
\newcommand{\D}{\mathbb{D}}
\newcommand{\De}{\mathbb{D}_{\mathrm{e}}}

\renewcommand{\P}{\mathbb{CP}}
\newcommand{\T}{\mathbb{T}}
\newcommand{\Chat}{\widehat{\C}}
\newcommand{\caplog}{\operatorname{cap}}
\newcommand{\supp}{\operatorname{supp}}
\newcommand{\Span}{\operatorname{span}}
\newcommand{\ord}{\operatorname{ord}}

\title[Fermat's Last Theorem in star-invariant subspaces]{Fermat's Last Theorem in star-invariant subspaces:  capacity-zero boundary spectrum}
\author[S.R.~Garcia]{Stephan Ramon Garcia}
\address{Department of Mathematics and Statistics, Pomona College, 610 N. College Ave., Claremont, CA 91711, USA}
\email{stephan.garcia@pomona.edu}
\urladdr{\url{https://stephangarcia.sites.pomona.edu/}}

\thanks{Partially supported by NSF grant DMS-2452084.}


\begin{document}

\maketitle

\begin{abstract}
We prove that the Fermat equation $f^n+g^n=h^n$ with exponent $n\geq 4$ has no nontrivial solutions $f,g,h$ in the star-invariant subspace $K_\theta^p$, in which $1\leq p\leq\infty$, whenever the boundary spectrum of the inner function $\theta$ has logarithmic capacity zero.  This gives a positive answer, for this class of inner functions and exponents, to a problem posed by Dyakonov.
\end{abstract}

\section{Introduction}\label{Section:Introduction}
This paper concerns an intriguing problem proposed by Dyakonov over a decade ago \cite{Dyakonov}.
Despite its simple statement and evident appeal, it has not been well studied.
To be more specific, he asked whether Fermat's Last Theorem (FLT) holds for star-invariant subspaces of Hardy spaces.

Before stating Dyakonov's problem, we require some definitions and notation.
Let $H^p$ denote the Hardy space on the open unit disk $\D$, in which $1\leq p\leq\infty$,
and let $\theta$ denote an inner function \cite{Duren, Koosis, Garnett}.  The associated \emph{star-invariant subspace}
is the subspace $K_{\theta}^p = H^p \cap \theta \overline{zH^p}$ of $H^p$, in which
we have identified $H^p$ functions with their almost-everywhere defined nontangential boundary functions
on the unit circle $\T$.  These are invariant subspaces for the backward shift on $H^p$ \cite{CimaRoss, DSS, CBSTK}.
The spaces $K_{\theta}^2$ are of particular importance; they are the so-called \emph{model spaces} and play a central role
in the model theory for contractions \cite{IMSTO, MS, MSAS, Nikolski}.

A triple $(f,g,h)$ of functions in $K_{\theta}^p$ is \emph{trivial} if $\dim\Span\{f,g,h\}\leq1$ and it is \emph{nontrivial} otherwise.
In 2013, Dyakonov posed the following problem in \cite[Prob.~2]{Dyakonov}:
\begin{quote}\em
Prove or disprove ``Fermat's last theorem'' for $K_{\theta}^p$: whenever $\theta$ is inner, $p \in [1,\infty]$
and $n \in \{3,4,5,\ldots\}$, no nontrivial triple of functions $f,g,h \in K_{\theta}^p$ satisfies $f^n+g^n=h^n$.
\end{quote}
Since $K_{z^3}^p = \Span\{1,z,z^2\}$ for $1 \leq p \leq \infty$ and 
$(1-z^2)^2+(2z)^2=(1+z^2)^2$, the restriction $n \geq 3$ is necessary to make things interesting.

Although his article \cite{Dyakonov} appeared in 2014, Dyakonov posed his question at least as early as
the conference  ``Invariant subspaces of the shift operator'' at the Centre de Recherches Math\'ematiques in August 2013. During an open problem session, he said something to the effect:  ``the next problem is Fermat's Last Theorem,'' followed by a long pause which was met with a few chuckles.  Most of the audience thought it was a joke.  When Dyakonov added ``for model spaces,'' a quiet fell upon the room, for the audience instantly understood that it was a nontrivial problem of a highly novel nature. He explained how the Mason--Stothers theorem (the polynomial $abc$-theorem) handles finite-dimensional model spaces. This theorem states that if $a,b,c$ are relatively prime polynomials, not all constant, and $a+b=c$, then $\max\{\deg a, \deg b, \deg c\}$ is less than the number of distinct zeros of $abc$ in $\C$ \cite{Mason, Stothers}.

To our knowledge, there has been no work on the Dyakonov's problem beyond his original paper. This is not due to lack of interest, but rather because of the difficulty of the problem.  One cannot expect anything related to Fermat's Last Theorem to be easy.  The problem was advertised again in 2016 \cite[p.~124]{IMSTO}.

The main result of this paper proves FLT for exponents $n \geq 4$ and model spaces for which the spectrum of the inner function meets the unit circle in a small set, in the sense of logarithmic capacity.
Since this article is aimed at the operator-related function theorist,
we assume familiarity with inner functions \cite{Duren, Koosis, Garnett}, star-invariant subspaces \cite{CimaRoss,IMSTO,Nikolski}, logarithmic capacity \cite{Ransford, SaffSurvey, SaffTotik}, and Riemann surfaces \cite{Donaldson, Farkas, Miranda, Gamelin, Jost}.

Let $\theta$ be an inner function and write 
$\theta=\alpha B S$, in which $\alpha \in \T$, $B$ is a Blaschke product with zero sequence $Z(B)$,
and $S$ is a singular inner function with finite positive singular measure $\mu$ on $\T$.
The \emph{spectrum} of $\theta$ is $\sigma(\theta) = Z(B)^- \cup \supp \mu$;
that is, the union of the closure of its zero set and the support of $\mu$ \cite[Def.~7.3.2]{IMSTO}.
The term ``spectrum'' arises because it is the spectrum of the associated compressed shift operator \cite[Thm.~9.6.1]{IMSTO}.  However, our interest in 
the spectrum of $\theta$ arises from its connection with holomorphic continuation: $\sigma(\theta)\cap \T$
is the set of points of $\T$ across which $\theta$ cannot be holomorphically continued
\cite[Theorem 7.3.4.i \& Prop.~7.3.8]{IMSTO}.

Our main result appears to be the first progress on Dyakonov's problem since 2013, when he first proposed it
and handled the finite-dimensional case \cite{Dyakonov}, in which the inner function is a finite Blaschke product \cite{FBP, FBPAS}.

\begin{theorem}\label{Theorem:Main}
Let $1 \leq p \leq \infty$, let $\theta$ be an inner function such that $\caplog (\sigma(\theta)\cap\T)=0$.  
If $n \geq 4$ and $f,g,h \in K_{\theta}^p$ satisfy $f^n+g^n=h^n$, then $\dim\Span\{f,g,h\} \leq 1$.
\end{theorem}

For example, the theorem applies when $\theta$ is an atomic inner function or a Blaschke product whose zeros accumulate only on a countable set.

Dyakonov remarks in \cite{Dyakonov} 
that ``[I]t seems that a suitable $K_{\theta}^p$ version of the Mason--Stothers theorem, if
available, might give us a clue to Problem 2.''  Although this may yet be the case, we take a different approach here.
Instead of a zero-counting approach or an $abc$-analogue,
we employ pseudocontinuations, properties of the projective Fermat curve, Riemann surfaces, the Uniformization Theorem, and a hyperbolic extension theorem
of Nishino--Suzuki--J\"arvi \cite{Nishino, Suzuki, Jarvi}.
Thus, our results combine operator-related function theory going back to Douglas--Shapiro--Shields \cite{DSS}
with techniques more familiar in number theory and algebraic geometry.

\begin{remark}
Gross proved that $f^n+g^n=1$ has no nonconstant meromorphic solutions for $n \geq 4$, whereas for $n=3$ 
solutions can be constructed using Weierstrass' $\wp$ function \cite{Gross1966,GrossFunctional1966, Gross1968}. Baker \cite{Baker1966} later described all meromorphic solutions of the cubic equation in terms of these elliptic solutions and composition with entire functions. These results do not apply directly to our problem since functions in star-invariant subspaces need not extend meromorphically to $\C$ if $\sigma(\theta) \cap \T \neq \varnothing$.  We are able to overcome this obstruction by working projectively.
\end{remark}

\begin{remark}
The author, who coauthored a book on model spaces \cite{IMSTO} and wrote his first paper on star-invariant subspaces \cite{CBSTK}, 
has published many papers in number theory, including a few tangentially related to
Fermat's Last Theorem \cite{SAHS, K6M}.  Consequently, Dyakonov's version of the problem had long fascinated him and the author cannot suppress his natural enthusiasm at finally having connected his two seemingly disparate interests.  However, he must acknowledge fruitful discussions with ChatGPT, which located crucial references which finally permitted the author to make some progress on this problem after many years of unsuccessful contemplation. After the initial draft was complete, Claude was then used to proofread and criticize this note.  Any remaining mistakes are, of course, the fault of the author.
\end{remark}

This brief note is organized as follows.
Section \ref{Section:Preliminaries} contains preliminary material about the projective Fermat curve and hyperbolic rigidity
that may be unfamiliar to certain readers.
The proof of Theorem \ref{Theorem:Main} appears in Section \ref{Section:Proof}.

\section{Preliminaries}\label{Section:Preliminaries}

Two-dimensional complex projective space is $\P^2 = (\C^3 \setminus \{\mathbf{0}\}) / \sim$, in which $\sim$ identifies nonzero scalar multiples \cite[Ch.~4]{Fulton}. The equivalence class of $(z_0,z_1,z_2)$ is denoted $[z_0 : z_1 : z_2]$.  One can show that $\P^2$ is a compact, complex manifold of complex dimension $2$.
Operator-related function theorists will be familiar with the analogous projective complex line $\P^1$: it is the Riemann sphere.

The universal covering surface of a compact, connected Riemann surface is, up to conformal equivalence, 
$\Chat = \C \cup \{\infty\}$ if $g(X) = 0$, $\C$ if $g(X) = 1$, or $\D$ if $g(X) \geq 2$, in which $g(X)$ denotes the genus of $X$ \cite[Ch.~10, Thm.~12]{Donaldson}, \cite[p.~443]{Gamelin}, \cite[Thm.~4.4.1 \& Cor.~4.4.1]{Jost}.  Loosely put, $g(X)$ counts the numbers of handles in $X$.  For example,
the Riemann sphere has genus $0$ and the torus has genus $1$.

The \emph{projective Fermat curve}
\begin{equation}\label{eq:FermatCurve}
X_n = \big\{ [x,y,z] \in \P^2 : x^n+y^n=z^n\big\}
\end{equation}
is a smooth, compact, connected Riemann surface with genus
\begin{equation}\label{eq:FermatGenus}
g(X_n) = \frac{1}{2}(n-1)(n-2);
\end{equation}
see \cite[Ch.~3, Prop.~3 \& Ch.~7, Prop.~20]{Donaldson}.
Thus, the universal covering surface of $X_n$ is $\D$ for $n \geq 4$.  Note that $g(X_3) = 1$, so $X_3$ has $\C$
as its universal covering surface.  This distinction prevents Lemmas \ref{Lemma:Extension} and \ref{Lemma:HyperbolicRigidity} below from applying when $n=3$.

The defining polynomial $x^n+y^n-z^n$ is homogeneous of degree $n$, so the cancellation of a common local factor in homogeneous coordinates preserves the Fermat relation. Thus, common zeros and poles are removable \cite[Thm.~1.8]{Forster}.

\begin{lemma}\label{Lemma:Removable}
Let $\Omega$ be a Riemann surface, let $n \geq 2$, and let
$f,g,h$ be meromorphic functions on $\Omega$, not all identically zero, such that
$f^n+g^n=h^n$. Then $\Phi=[f:g:h]$ determines a holomorphic map
$\Phi:\Omega \to X_n$; that is, any common zeros of $f,g,h$ and any poles of $f,g,h$ are projectively removable.
\end{lemma}

\begin{proof}
Fix $a\in\Omega$ and fix a local coordinate $z$ centered at $a$. Let
\begin{equation*}
m = \min\{ \ord_a f, \ord_a g, \ord_a h\},
\end{equation*}
in which a pole has negative order and the zero function has order $\infty$.  
Since $f,g,h$ are not all identically zero, $m$ is finite.  On a neighborhood of $a$, we have
\begin{equation}
F(z)=z^{-m}f(z), \qquad
G(z)=z^{-m}g(z), \quad \text{and} \quad
H(z)=z^{-m}h(z),
\end{equation}
in which $F$, $G$, and $H$ are holomorphic near $a$ and at least one of
$F(a)$, $G(a)$, and $H(a)$ is nonzero.  
Thus, $[F:G:H]$ is a holomorphic map into $\P^2$ on a neighborhood of $a$.
Since the Fermat equation is homogeneous of degree $n$, we get
\begin{equation}
F^n+G^n-H^n = z^{-mn}(f^n+g^n-h^n) = 0
\end{equation}
on that neighborhood, so $[F:G:H]$ maps into $X_n$.  Away from $a$, we have
$[F:G:H] = [f:g:h]$, which yields a holomorphic extension across $a$.
Since $a \in \Omega$ was arbitrary and because the local holomorphic extensions agree on overlaps, they
combine to give a holomorphic map $\Phi:\Omega\to X_n$.
\end{proof}

The next lemma is due to Nishino \cite[Thm.~II, p.110-111]{Nishino}. 
Suzuki later gave a simpler proof \cite{Suzuki}.  J\"arvi then placed the result in the context of Picard-type extension theorems for Riemann surfaces\cite[Cor.~1]{Jarvi}.

\begin{lemma}\label{Lemma:Extension}
Let $\Omega\subset\C$ be a domain, let $E\subset \Omega$ be compact with
$\caplog E=0$, and let $X$ be a compact Riemann surface of genus at least two.  Every holomorphic map
$\Phi:\Omega\setminus E\to  X$ extends holomorphically to $\Omega$.
\end{lemma}

Lastly, we require the following hyperbolic rigidity theorem \cite[Thm.~2.6.3]{Jost}.

\begin{lemma}\label{Lemma:HyperbolicRigidity}
Let $X$ be a compact Riemann surface of genus at least two.  Every holomorphic map
$\Phi:\C\to  X$ is constant.
\end{lemma}

The proof is simple: the universal covering surface of $X$ is $\D$.  Since $\C$ is simply connected, $\Phi$ lifts to a holomorphic 
map $\widetilde{\Phi}:\C\to \D$.  The lift is a bounded entire function, so Liouville's theorem ensures that it is constant.  Thus, $\Phi$ is constant.  

\section{Proof of Theorem \ref{Theorem:Main}}\label{Section:Proof}
Let $\theta$ be an inner function and define $E=\sigma(\theta)\cap\T$.
Suppose that $f,g,h\in K_{\theta}^p$ satisfy $f^n+g^n=h^n$
for some $n \geq 4$ and that $\caplog E=0$.  We must prove that
$\dim\Span\{f,g,h\}\leq1$.  If $f=g=h=0$, there is nothing to prove,
so we assume otherwise.
We identify functions in $K_{\theta}^p$ with their a.e.~defined boundary values on $\T$.
The map $f\mapsto \overline{fz}\theta$ is a conjugation on $K_{\theta}^p$: it is 
conjugate-linear, isometric, and involutive \cite{CBSTK}.  Thus, there are $\widetilde{f},\widetilde{g},\widetilde{h} \in K_{\theta}^p$ such that
\begin{equation}\label{eq:fgh}
f=\theta\overline{z\widetilde{f}}, \qquad
g=\theta\overline{z\widetilde{g}}, \quad\text{and}\quad
h=\theta\overline{z\widetilde{h}}
\end{equation}
a.e.~on $\T$. Since $f,g,h$ satisfy the Fermat equation, we get
\begin{equation}
(\overline{z}\theta)^n
\overline{ \big( (\widetilde{f})^n + (\widetilde{g})^n - (\widetilde{h})^n \big)} = 0
\end{equation}
a.e.~on $\T$ and hence (by uniqueness of boundary values for Smirnov functions)
\begin{equation}\label{eq:ConjugateFermat}
(\widetilde{f})^n + (\widetilde{g})^n = (\widetilde{h})^n
\end{equation}
on $\D$.  That is, the conjugate triple also satisfies the Fermat equation.

If $q$ is holomorphic on $\D$, then its \emph{reflection} through $\T$ is the holomorphic function
\begin{equation}
q^\#(z) = \overline{q(1/\overline z)}
\end{equation}
on $\De = \{z:|z|>1\} \cup \{\infty\}$.  
Reflect \eqref{eq:ConjugateFermat} and deduce that
\begin{equation}
(\widetilde{f}^\#)^n + (\widetilde{g}^\#)^n = (\widetilde{h}^\#)^n
\end{equation}
on $\De$.  Therefore,
$\Phi_{\mathrm{e}} = [\widetilde{f}^\#:\widetilde{g}^\#:\widetilde{h}^\#] : \De \to X_n$
is holomorphic after one removes common zeros projectively by Lemma \ref{Lemma:Removable}. 
Similarly, $\Phi_{\mathrm{i}}=[f:g:h]:\D\to X_n$ is holomorphic after removing common zeros.

Let $I \subseteq \T\setminus E$ be an open arc.  Then $\theta$
extends holomorphically and without zeros across $I$ \cite[Theorem 7.3.4.i \& Prop.~7.3.8]{IMSTO}.  
Then \eqref{eq:fgh} ensures that 
\begin{equation}
f=\frac{\theta}{z}\widetilde{f}^\#,\qquad
g=\frac{\theta}{z}\widetilde{g}^\#,\quad\text{and}\quad
h=\frac{\theta}{z}\widetilde{h}^\#
\end{equation}
on a neighborhood of $I$.  Since $\theta/z$ is nonvanishing on $I$, we get
\begin{equation*}
[f:g:h] = [\widetilde{f}^\#:\widetilde{g}^\#:\widetilde{h}^\#]
\end{equation*}
there. Thus, $\Phi_{\mathrm{i}}$ and $\Phi_{\mathrm{e}}$ agree across every
open arc of $\T\setminus E$ and hence they glue to a holomorphic map from $\Chat\setminus E$ to $X_n$.
We consider the restricted map $\Phi:\C\setminus E\to X_n$.

Since $n \geq 4$, \eqref{eq:FermatGenus} ensures that $g(X_n) \geq 2$.
Since $\caplog E=0$, Lemma \ref{Lemma:Extension} provides a holomorphic extension
$\widetilde{\Phi}:\C\to X_n$ of $\Phi$ across $E$.
Then Lemma \ref{Lemma:HyperbolicRigidity} ensures that $\widetilde{\Phi}$ is constant.
Thus, the holomorphic projective map determined by $f,g,h$ is constant,
say with value $[\alpha:\beta:\gamma]\in X_n$.
If $\gamma \neq 0$, then
$f/h = \alpha / \gamma$ and $g/h = \beta/\gamma$ wherever
$h$ does not vanish, so the Identity Theorem ensures that $f = (\alpha/\gamma)h$ and $g = (\beta/\gamma)h$
on $\D$, so $\dim \Span\{f,g,h\} \leq 1$.  The other two cases are similar. \qed

\bibliography{FLTSIS}

\providecommand{\bysame}{\leavevmode\hbox to3em{\hrulefill}\thinspace}
\providecommand{\MR}{\relax\ifhmode\unskip\space\fi MR }
\providecommand{\MRhref}[2]{%
  \href{http://www.ams.org/mathscinet-getitem?mr=#1}{#2}
}
\providecommand{\href}[2]{#2}
\begin{thebibliography}{10}

\bibitem{Baker1966}
I.~N. Baker, \emph{On a class of meromorphic functions}, Proc. Amer. Math. Soc.
  \textbf{17} (1966), 819--822. \MR{197732}

\bibitem{CimaRoss}
Joseph~A. Cima and William~T. Ross, \emph{The backward shift on the {H}ardy
  space}, Mathematical Surveys and Monographs, vol.~79, American Mathematical
  Society, Providence, RI, 2000. \MR{1761913}

\bibitem{Donaldson}
Simon Donaldson, \emph{Riemann surfaces}, Oxford Graduate Texts in Mathematics,
  vol.~22, Oxford University Press, Oxford, 2011. \MR{2856237}

\bibitem{DSS}
R.~G. Douglas, H.~S. Shapiro, and A.~L. Shields, \emph{Cyclic vectors and
  invariant subspaces for the backward shift operator}, Ann. Inst. Fourier
  (Grenoble) \textbf{20} (1970), no.~fasc. 1, 37--76. \MR{270196}

\bibitem{Duren}
Peter~L. Duren, \emph{Theory of {$H^{p}$} spaces}, Pure and Applied
  Mathematics, Vol. 38, Academic Press, New York-London, 1970. \MR{268655}

\bibitem{Dyakonov}
Konstantin~M. Dyakonov, \emph{Two problems on coinvariant subspaces of the
  shift operator}, Integral Equations Operator Theory \textbf{78} (2014),
  no.~2, 151--154. \MR{3157976}

\bibitem{Farkas}
H.~M. Farkas and I.~Kra, \emph{Riemann surfaces}, second ed., Graduate Texts in
  Mathematics, vol.~71, Springer-Verlag, New York, 1992. \MR{1139765}

\bibitem{Forster}
Otto Forster, \emph{Lectures on {R}iemann surfaces}, Graduate Texts in
  Mathematics, vol.~81, Springer-Verlag, New York, 1991, Translated from the
  1977 German original by Bruce Gilligan, Reprint of the 1981 English
  translation. \MR{1185074}

\bibitem{Fulton}
William Fulton, \emph{Algebraic curves}, Advanced Book Classics, Addison-Wesley
  Publishing Company, Advanced Book Program, Redwood City, CA, 1989, An
  introduction to algebraic geometry, Notes written with the collaboration of
  Richard Weiss, Reprint of 1969 original. \MR{1042981}

\bibitem{Gamelin}
Theodore~W. Gamelin, \emph{Complex analysis}, Undergraduate Texts in
  Mathematics, Springer-Verlag, New York, 2001. \MR{1830078}

\bibitem{CBSTK}
Stephan~Ramon Garcia, \emph{Conjugation, the backward shift, and {T}oeplitz
  kernels}, J. Operator Theory \textbf{54} (2005), no.~2, 239--250.
  \MR{2186351}

\bibitem{MS}
\bysame, \emph{Model spaces}, Lectures on analytic function spaces and their
  applications, Fields Inst. Monogr., vol.~39, Springer, Cham, [2023]
  \copyright 2023, pp.~121--154. \MR{4676335}

\bibitem{SAHS}
Stephan~Ramon Garcia and Bob Lutz, \emph{A supercharacter approach to
  {H}eilbronn sums}, J. Number Theory \textbf{186} (2018), 1--15. \MR{3758203}

\bibitem{IMSTO}
Stephan~Ramon Garcia, Javad Mashreghi, and William~T. Ross, \emph{Introduction
  to model spaces and their operators}, Cambridge Studies in Advanced
  Mathematics, vol. 148, Cambridge University Press, Cambridge, 2016.
  \MR{3526203}

\bibitem{FBPAS}
\bysame, \emph{Finite {B}laschke products: a survey}, Harmonic analysis,
  function theory, operator theory, and their applications, Theta Ser. Adv.
  Math., vol.~19, Theta, Bucharest, 2017, pp.~133--158. \MR{3753897}

\bibitem{FBP}
\bysame, \emph{Finite {B}laschke products and their connections}, Springer,
  Cham, 2018. \MR{3793610}

\bibitem{MSAS}
Stephan~Ramon Garcia and William~T. Ross, \emph{Model spaces: a survey},
  Invariant subspaces of the shift operator, Contemp. Math., vol. 638, Amer.
  Math. Soc., Providence, RI, 2015, pp.~197--245. \MR{3309355}

\bibitem{K6M}
Stephan~Ramon Garcia and George Todd, \emph{Supercharacters, elliptic curves,
  and the sixth moment of {K}loosterman sums}, J. Number Theory \textbf{202}
  (2019), 316--331. \MR{3958076}

\bibitem{Garnett}
John~B. Garnett, \emph{Bounded analytic functions}, first ed., Graduate Texts
  in Mathematics, vol. 236, Springer, New York, 2007. \MR{2261424}

\bibitem{Gross1966}
Fred Gross, \emph{On the equation {$f^{n}+g^{n}=1$}}, Bull. Amer. Math. Soc.
  \textbf{72} (1966), 86--88. \MR{185125}

\bibitem{GrossFunctional1966}
\bysame, \emph{On the functional equation {$f^{n}+g^{n}=h^{n}$}}, Amer. Math.
  Monthly \textbf{73} (1966), 1093--1096. \MR{204655}

\bibitem{Gross1968}
\bysame, \emph{On the equation {$f^{n}+g^{n}=1$}. {II}}, Bull. Amer. Math. Soc.
  \textbf{74} (1968), 647--648. \MR{227406}

\bibitem{Jarvi}
Pentti J\"{a}rvi, \emph{Generalizations of {P}icard's theorem for {R}iemann
  surfaces}, Trans. Amer. Math. Soc. \textbf{323} (1991), no.~2, 749--763.
  \MR{1030508}

\bibitem{Jost}
J\"{u}rgen Jost, \emph{Compact {R}iemann surfaces}, third ed., Universitext,
  Springer-Verlag, Berlin, 2006, An introduction to contemporary mathematics.
  \MR{2247485}

\bibitem{Koosis}
Paul Koosis, \emph{Introduction to {$H_p$} spaces}, second ed., Cambridge
  Tracts in Mathematics, vol. 115, Cambridge University Press, Cambridge, 1998,
  With two appendices by V. P. Havin [Viktor Petrovich Khavin]. \MR{1669574}

\bibitem{Mason}
R.~C. Mason, \emph{Diophantine equations over function fields}, London
  Mathematical Society Lecture Note Series, vol.~96, Cambridge University
  Press, Cambridge, 1984. \MR{754559}

\bibitem{Miranda}
Rick Miranda, \emph{Algebraic curves and {R}iemann surfaces}, Graduate Studies
  in Mathematics, vol.~5, American Mathematical Society, Providence, RI, 1995.
  \MR{1326604}

\bibitem{Nikolski}
N.~K. Nikolski\u{\i}, \emph{Treatise on the shift operator}, Grundlehren der
  mathematischen Wissenschaften [Fundamental Principles of Mathematical
  Sciences], vol. 273, Springer-Verlag, Berlin, 1986, Spectral function theory,
  With an appendix by S. V. Hru\v{s}\v{c}ev [S. V. Khrushch\"{e}v] and V. V.
  Peller, Translated from the Russian by Jaak Peetre. \MR{827223}

\bibitem{Nishino}
Toshio Nishino, \emph{Prolongements analytiques au sens de {R}iemann}, Bull.
  Soc. Math. France \textbf{107} (1979), no.~1, 97--112. \MR{532563}

\bibitem{Ransford}
Thomas Ransford, \emph{Potential theory in the complex plane}, London
  Mathematical Society Student Texts, vol.~28, Cambridge University Press,
  Cambridge, 1995. \MR{1334766}

\bibitem{SaffSurvey}
E.~B. Saff, \emph{Logarithmic potential theory with applications to
  approximation theory}, Surv. Approx. Theory \textbf{5} (2010), 165--200.
  \MR{2734198}

\bibitem{SaffTotik}
Edward~B. Saff and Vilmos Totik, \emph{Logarithmic potentials with external
  fields}, Grundlehren der mathematischen Wissenschaften [Fundamental
  Principles of Mathematical Sciences], vol. 316, Springer, Cham, [2024]
  \copyright 2024, Second edition [of 1485778], With an appendix by Thomas
  Bloom. \MR{4807484}

\bibitem{Stothers}
W.~W. Stothers, \emph{Polynomial identities and {H}auptmoduln}, Quart. J. Math.
  Oxford Ser. (2) \textbf{32} (1981), no.~127, 349--370. \MR{625647}

\bibitem{Suzuki}
Masakazu Suzuki, \emph{Comportement des applications holomorphes autour d'un
  ensemble polaire}, C. R. Acad. Sci. Paris S\'{e}r. I Math. \textbf{304}
  (1987), no.~8, 191--194. \MR{880924}

\end{thebibliography}
\bibliographystyle{amsplain}

\end{document}